\documentclass[12pt,a4paper]{article}

\usepackage{geometry}
\usepackage{graphicx,titling}%
\usepackage{multirow}%
\usepackage[title]{appendix}%
\usepackage{xcolor}%
\usepackage{textcomp}%
\usepackage{manyfoot}%
\usepackage{booktabs}%

\usepackage{listings}%
\usepackage{multirow}%
\usepackage{amsmath,amssymb,amsfonts,dsfont}%
\usepackage{amsthm}%
\usepackage{mathrsfs}%
\usepackage[title]{appendix}%
\usepackage{tikz}%
\usepackage{pgfplots}
\usepackage{float}
\usepackage{caption}
\usepackage{subcaption}
\usepackage{color}
\usepackage{cleveref}
\usepackage{placeins}
\usepackage{bbm}
\usepackage{scalerel}
\usepackage{tikz-cd}

\usepackage[ruled,linesnumbered]{algorithm2e}

\RestyleAlgo{ruled}

\newcommand{\Ker}[1]{\operatorname{Ker}\left(#1\right)}
\newcommand{\Cok}[1]{\operatorname{Cok}\left(#1\right)}
\newcommand{\Img}[1]{\operatorname{Im}\left(#1\right)}
\newcommand{\A}{\mathbb{A}}
\newcommand{\B}{\mathbb{B}}
\newcommand{\X}{\mathbb{X}}
\newcommand{\R}{\mathbb{R}}
\newcommand{\Z}{\mathbb{Z}}
\DeclareMathOperator{\spa}{Span}

\theoremstyle{plain}%
\newtheorem{theorem}{Theorem}[section]%
\newtheorem{corollary}[theorem]{Corollary}

 \theoremstyle{definition}%
\newtheorem{example}{Example}%
\newtheorem{definition}{Definition}%

\begin{document}

\title{Identifying cobordisms using kernel persistence}

\author{Y. B. Bleile, L. Fajstrup, T. Heiss, A. M. Svane, S. S. Sørensen }
\date{}

\maketitle
\begin{abstract}{Motivated by applications in chemistry, we give a homlogical definition of  tunnels, or more generally cobordisms, connecting disjoint parts of a cell complex. For a filtered complex, this defines a persistence module. We give a method for identifying birth and death times using kernel persistence and  a matrix reduction algorithm for pairing birth and death times. 
}
\end{abstract}

\section{Introduction}

Matter is built up by connected atoms. Depending on elemental constituents, temperature, pressure, kinetics, and thermodynamics, the atoms form chemical structures \cite{s1}. At the nanoscale, this presents itself as bonds between atoms, yet these connections also form larger structures which for many inorganic solids are represented as networks of atoms on the scale of nano- to micrometres. It is well-known that the packing of the atoms into dense or void-filled structures directly affect physical and chemical properties of materials \cite{s2}. This is particularly relevant in applications where atoms need to move, e.g., solid state battery materials where lithium ions travel through “tunnels” in a fixed network or gas adsorption where e.g.\ CO$_2$ needs to be adsorbed to a surface within a porous material \cite{s3,s5,s4}. However, identifying and characterizing packing of atoms beyond the very nearest neighbours, for example tunnels in the void space where atoms can move through, is inherently difficult. This is particularly true in materials which are non-crystalline (i.e., they lack the atomic periodicity from crystals exploited e.g.\ in \cite{crystal}) such as glasses. To overcome this, the development of new methods is needed.  A previous approach applied a reeb graph reconstruction of the space between atoms \cite{matteo}, which, however, was computationally expensive. Persistent homology, has the advantage of being computationally efficient and has previously proven useful for characterizing glass materials \cite{s7,s6}. Thus, in this paper, we take a homological approach to modelling tunnels between atoms.

We model the atoms of a material as a point cloud in a 3D box and look for tunnels through the void space  connecting the top and bottom of the box, see left part of Figure \ref{fig:intro} for a 2D-visualization. More precisely, if we place balls of radius $r$ around each atom, a tunnel of width $r$ will be a connected component in the complement of the balls intersected with the box that intersects both top and bottom of the box, see right part of Figure \ref{fig:intro}. Equivalently, such a tunnel can be represented by a 2-dimensional surface in the union of balls that surrounds the tunnel. We give a homological definition way of counting such tunnels. When the radius $r$ of the balls increases, such 
tunnels appear and disappear again --- they are born and die again.  The birth time of a tunnel is interpreted as the first radius where the tunnel becomes separated from an existing tunnel or from the unbounded component except for the opening at the top and the bottom.  The death time is the first radius where there is no longer a passage through the tunnel, 
corresponding to the ``bottleneck thickness'' of the tunnel, i.e.\ the largest radius of a ball that is able to roll through the tunnel.
We give an algorithm for computing pairs of birth and death times closely related to kernel persistence \cite{kernel-persistence}. The algorithm will be set up more generally for detecting $(k+1)$-dimensional open cobordism-like structures connecting disjoint subcomplexes of a filtered cell complex. 

The paper is structured as follows. We define a homological way of representing cobordisms inside a cell complex 
in \Cref{sec:tunnels}. Then, in \Cref{sec:persistence}, we consider the situation where the cell complex is filtered and define a persistence module that captures the births and deaths of these open cobordisms. The birth and death times can be identified using kernel persistence \cite{kernel-persistence}. However, pairing the birth and death times requires a new algorithm, which is presented in \Cref{sec:alg}. 

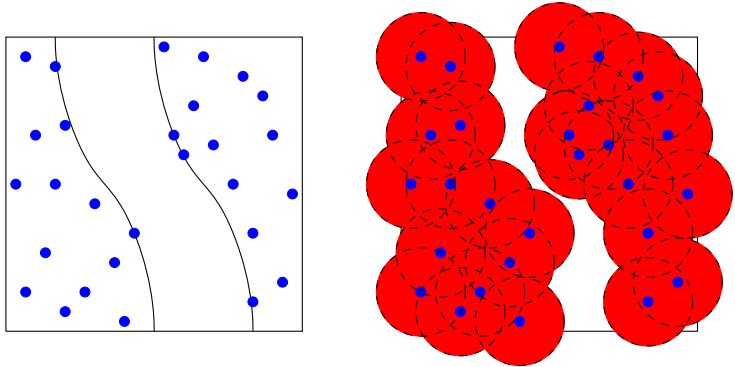
\begin{figure}
\begin{center}
\begin{tikzpicture}[scale=1.3]
\draw (0.5,0) -- (3.5,0) -- (3.5,3) -- (0.5,3) -- (0.5,0);
\draw plot [smooth, tension=1] coordinates { (2,0) (1.8,1) (1.2,2) (1,3)};
\draw plot [smooth, tension=1] coordinates { (3,0) (2.8,1) (2.2,2) (2,3)};
\draw[fill,blue]  (1.8,1) circle (0.05);
\draw[fill, blue]  (2.2,2) circle (0.05);
\draw[fill,blue]  (3,0.3) circle (0.05);
\draw[fill,blue]  (1,2.7) circle (0.05);

\draw[fill,blue]  (0.6,1.5) circle (0.05);
\draw[fill,blue]  (0.7,0.4) circle (0.05);
\draw[fill,blue]  (0.8,2) circle (0.05);
\draw[fill,blue]  (0.9,0.8) circle (0.05);
\draw[fill,blue]  (1,1.5) circle (0.05);
\draw[fill,blue]  (1.1,0.2) circle (0.05);
\draw[fill,blue]  (1.3,0.4) circle (0.05);
\draw[fill,blue]  (1.4,1.3) circle (0.05);
\draw[fill,blue]  (1.6,0.7) circle (0.05);
\draw[fill,blue]  (1.7,0.1) circle (0.05);
\draw[fill,blue]  (0.7,2.8) circle (0.05);
\draw[fill,blue]  (1.1,2.1) circle (0.05);

\draw[fill,blue]  (2.1,2.9) circle (0.05);
\draw[fill,blue]  (2.4,2.3) circle (0.05);
\draw[fill,blue]  (2.3,1.8) circle (0.05);
\draw[fill,blue]  (2.5,2.8) circle (0.05);
\draw[fill,blue]  (2.6,1.9) circle (0.05);
\draw[fill,blue]  (2.9,2.6) circle (0.05);
\draw[fill,blue]  (2.8,1.5) circle (0.05);
\draw[fill,blue]  (3.1,2.4) circle (0.05);
\draw[fill,blue]  (3,1) circle (0.05);
\draw[fill,blue]  (3.2,2) circle (0.05);
\draw[fill,blue]  (3.3,0.5) circle (0.05);
\draw[fill,blue]  (3.4,1.4) circle (0.05);

\draw (4.5,0) -- (7.5,0) -- (7.5,3) -- (4.5,3) -- (4.5,0);

\draw[fill,red]  (5.8,1) circle (0.45);
\draw[fill,red]  (6.2,2) circle (0.45);
\draw[fill,red]  (7,0.3) circle (0.45);
\draw[fill,red]  (5,2.7) circle (0.45);

\draw[fill,red]  (4.6,1.5) circle (0.45);
\draw[fill,red]  (4.7,0.4) circle (0.45);
\draw[fill,red]  (4.8,2) circle (0.45);
\draw[fill,red]  (4.9,0.8) circle (0.45);
\draw[fill,red]  (5,1.5) circle (0.45);
\draw[fill,red]  (5.1,0.2) circle (0.45);
\draw[fill,red]  (5.3,0.4) circle (0.45);
\draw[fill,red]  (5.4,1.3) circle (0.45);
\draw[fill,red]  (5.6,0.7) circle (0.45);
\draw[fill,red]  (5.7,0.1) circle (0.45);
\draw[fill,red]  (4.7,2.8) circle (0.45);
\draw[fill,red]  (5.1,2.1) circle (0.45);

\draw[fill,red]  (6.1,2.9) circle (0.45);
\draw[fill,red]  (6.4,2.3) circle (0.45);
\draw[fill,red]  (6.3,1.8) circle (0.45);
\draw[fill,red]  (6.5,2.8) circle (0.45);
\draw[fill,red]  (6.6,1.9) circle (0.45);
\draw[fill,red]  (6.9,2.6) circle (0.45);
\draw[fill,red]  (6.8,1.5) circle (0.45);
\draw[fill,red]  (7.1,2.4) circle (0.45);
\draw[fill,red]  (7,1) circle (0.45);
\draw[fill,red]  (7.2,2) circle (0.45);
\draw[fill,red]  (7.3,0.5) circle (0.45);
\draw[fill,red]  (7.4,1.4) circle (0.45);

\draw[fill,blue]  (5.8,1) circle (0.05);
\draw[fill,blue]  (6.2,2) circle (0.05);
\draw[fill,blue]  (7,0.3) circle (0.05);
\draw[fill,blue]  (5,2.7) circle (0.05);

\draw[fill,blue]  (4.6,1.5) circle (0.05);
\draw[fill,blue]  (4.7,0.4) circle (0.05);
\draw[fill,blue]  (4.8,2) circle (0.05);
\draw[fill,blue]  (4.9,0.8) circle (0.05);
\draw[fill,blue]  (5,1.5) circle (0.05);
\draw[fill,blue]  (5.1,0.2) circle (0.05);
\draw[fill,blue]  (5.3,0.4) circle (0.05);
\draw[fill,blue]  (5.4,1.3) circle (0.05);
\draw[fill,blue]  (5.6,0.7) circle (0.05);
\draw[fill,blue]  (5.7,0.1) circle (0.05);
\draw[fill,blue]  (4.7,2.8) circle (0.05);
\draw[fill,blue]  (5.1,2.1) circle (0.05);

\draw[fill,blue]  (6.1,2.9) circle (0.05);
\draw[fill,blue]  (6.4,2.3) circle (0.05);
\draw[fill,blue]  (6.3,1.8) circle (0.05);
\draw[fill,blue]  (6.5,2.8) circle (0.05);
\draw[fill,blue]  (6.6,1.9) circle (0.05);
\draw[fill,blue]  (6.9,2.6) circle (0.05);
\draw[fill,blue]  (6.8,1.5) circle (0.05);
\draw[fill,blue]  (7.1,2.4) circle (0.05);
\draw[fill,blue]  (7,1) circle (0.05);
\draw[fill,blue]  (7.2,2) circle (0.05);
\draw[fill,blue]  (7.3,0.5) circle (0.05);
\draw[fill,blue]  (7.4,1.4) circle (0.05);

\draw[dashed]  (5.8,1) circle (0.45);
\draw[dashed]  (6.2,2) circle (0.45);
\draw[dashed]  (7,0.3) circle (0.45);
\draw[dashed]  (5,2.7) circle (0.45);

\draw[dashed]  (4.6,1.5) circle (0.45);
\draw[dashed]  (4.7,0.4) circle (0.45);
\draw[dashed]  (4.8,2) circle (0.45);
\draw[dashed]  (4.9,0.8) circle (0.45);
\draw[dashed]  (5,1.5) circle (0.45);
\draw[dashed]  (5.1,0.2) circle (0.45);
\draw[dashed]  (5.3,0.4) circle (0.45);
\draw[dashed]  (5.4,1.3) circle (0.45);
\draw[dashed]  (5.6,0.7) circle (0.45);
\draw[dashed]  (5.7,0.1) circle (0.45);
\draw[dashed]  (4.7,2.8) circle (0.45);
\draw[dashed]  (5.1,2.1) circle (0.45);

\draw[dashed]  (6.1,2.9) circle (0.45);
\draw[dashed]  (6.4,2.3) circle (0.45);
\draw[dashed]  (6.3,1.8) circle (0.45);
\draw[dashed]  (6.5,2.8) circle (0.45);
\draw[dashed]  (6.6,1.9) circle (0.45);
\draw[dashed]  (6.9,2.6) circle (0.45);
\draw[dashed]  (6.8,1.5) circle (0.45);
\draw[dashed]  (7.1,2.4) circle (0.45);
\draw[dashed]  (7,1) circle (0.45);
\draw[dashed]  (7.2,2) circle (0.45);
\draw[dashed]  (7.3,0.5) circle (0.45);
\draw[dashed]  (7.4,1.4) circle (0.45);
\end{tikzpicture}
\end{center}
\caption{Left: tunnel through a point cloud connecting top and bottom. Right: Balls placed around points. The tunnel now corresponds to a path in the complement connecting top and bottom}\label{fig:intro}
\end{figure}

\section{A homological way of representing cobordisms}\label{sec:tunnels}
Let $\X$ be a finite cell complex, by which we will mean that $\X$ is a regular CW complex  with finitely many cells. 
Regularity means that not only the interior of each $k$-cell is homeomorphic to an open $k$-dimensional ball, but the homeomorphism also extends to the boundaries of the cell and the ball. \cite[Extended Appendix]{hatcher}. In particular, the boundary of each $k$-cell must be the union of $(k-1)$-dimensional cells.
Simplicial and cubical complexes \cite{cube} are examples of cell complexes. We will sometimes think of $\X$ as a topological space and sometimes as a combinatorial object hoping this causes no confusion.

         Let $\X$ be a cell complex, and take two disjoint non-empty subcomplexes $\A, \B \subseteq \X$. Then, $\A \cup \B$ is a non-empty subcomplex of $\X$. Consider the three maps on homology (with coefficients in $\Z/2\Z$ understood) induced by the inclusion maps
    \begin{align}\nonumber
        \iota^{\A}_{\ast}: {}&H_{*}\left( \A \right)\rightarrow H_{*}\left( \X\right), \\ \nonumber
        \iota^{\B}_{\ast}: {}&H_{*}\left( \B \right) \rightarrow H_{*}\left( \X \right), \\
        \iota^{\A \cup \B}_{\ast}: {}&H_{*}\left( \A \cup \B \right) = H_{*} \left( \A \right)\oplus H_{*} \left( \B \right) \rightarrow H_{*}\left( \X\right).\label{eq:sum}
    \end{align}
We can then define our main object of interest.
    
\begin{definition}[Cobordisms from $\A$ to $\B$]    
    Let $\X$ be a finite cell complex, and  $\A, \B \subseteq \X$ two disjoint non-empty subcomplexes. 
    Define the map
        \begin{equation}\label{eq:cokPhi}
            \Phi: \Ker{\iota^{\A}_{\ast}} \oplus \Ker{\iota^{\B}_{\ast}} \rightarrow \Ker{\iota^{\A \cup \B}_{\ast}}.
        \end{equation}
        We call $\Cok \Phi$ the \emph{equivalence classes of open cobordisms between $\A$ and $\B$}, and denote by $\Phi_k$ the restriction of $\Phi$ to degree $k$ in the grading.
    \end{definition}

To motivate our terminology, we give the following prototypical example. 

\begin{example}\label{ex:cylinder}
    Consider the triangulated cylinder $\X$ shown  in \Cref{fig:cylinder}. The vertices and edges are as shown in the figure with 12 trangles forming the sides of the cylinder, namely  
   \begin{align*}
        &v_0 v_1 v_4, v_0 v_2 v_5, v_0 v_3 v_4, v_0 v_3 v_5 , v_1 v_2 v_5, v_1 v_4 v_5,\\
        &v_3 v_4 v_7, v_3 v_5 v_8, v_3 v_6 v_7, v_3 v_6 v_8, v_4 v_5 v_8, v_4 v_7 v_8
    \end{align*} 
     Let $\A$ and $\B$ be the subcomplexes generated by the vertex sets $\{ v_0, v_1, v_2\}$ and $\{ v_6, v_7, v_8\}$, respectively. That is, $\A$ consists of the open triangle at the top (red), and $\B$ of the open triangle at the bottom (blue). For $k\in \{0,1\}$, we have the homology graoups $H_k(\X) \cong H_k(\A) \cong H_k(\B)\cong  \Z/2\Z$. Moreover,  $\Ker {\iota_k^\A} \cong \Ker {\iota_k^\B} \cong 0$ and $\Ker {\iota_k^{\A\cup \B}} = \spa \{c_1+c_2\}$, where $c_1$ and $c_2$ are the non-trivial elements of $H_k(\A)$ and $H_k(\B)$, respectively. It follows that 
     $$\Cok {\Phi_k} = \spa\{c_1+c_2\}.$$
     For $k=0$, this means that the two connected components in $\A\cup \B$ form a non-trival element of $\Cok {\Phi_0}$ corresponding to a path connecting them in $\X$. For $k=1$, 
     the union of the two circles in $\A\cup \B$ form a non-trivial element in $\Cok {\Phi_1}$ because the cylinder forms a 2-dimensional simplicial surface having their union as   boundary. 

If we insert the top triangle $v_0v_1v_2$ in $\A$ and consider $\Phi_1$, then $\Ker{\iota_*^\A}\cong H_1(\A)=0$ and  $\Ker{\iota_*^\B}\cong  \Ker{\iota_*^{\A\cup\B}}\cong \Z/2\Z$ because the blue cycle is now trivial in $\X$. Thus, $\Cok{\Phi_1}=0$. The cylinder is no longer counted because it is connects to a trivial cycle at the top. Similarly, if we instead inserted the middle triangle $v_3v_4v_5$, then $\Ker{\iota_*^\A}\cong \Ker{\iota_*^\B}\cong \Z/2\Z$ because both the top and bottom cycles are trivial in $\X$, but not in $\A$ and $\B$, respectively. Moreover,  $\Ker{\iota_*^{\A\cup\B}}\cong \Z/2\Z \oplus \Z/2\Z$ for the same reason, so $\Ker{\Phi_1}=0$. In both cases, inserting a triangle closed the tunnel making it trivial in $\Cok{\Phi_1}$.

    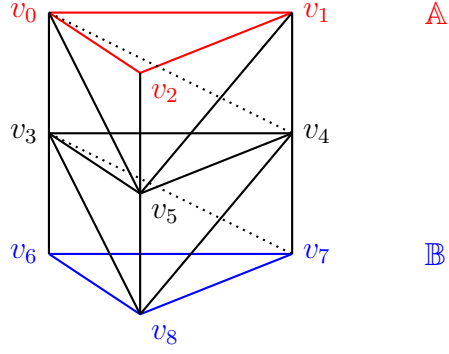
\begin{figure}[h!]
        \centering
        \begin{tikzpicture}[scale = 0.8]
            \draw[red] (-2,2) node[left] (0) {$v_0$};
            \draw[red] (2,2) node[right] (1) {$v_1$};
            \draw[red] (4,2) node[right] (9) {$\A$}; 
            \draw[red] (-0.5,1) node[below right] (2) {$v_2$};
            \draw (-2,0) node[left] (3) {$v_3$};
            \draw (2,0) node[right] (4) {$v_4$};
            \draw (-0.5,-1) node[below right] (5) {$v_5$};
            \draw[blue] (-2,-2) node[left] (6) {$v_6$};
            \draw[blue] (2,-2) node[right] (7) {$v_7$};
            \draw[blue] (4,-2) node[right] (10) {$\B$}; 
            \draw[blue] (-0.5,-3) node[below right] (8) {$v_8$};
            \draw[red, thick] (-2,2) -- (2,2);
            \draw[red, thick] (-2,2) -- (-0.5,1);
            \draw[red, thick] (-0.5,1) -- (2,2);
            \draw[black, thick] (-2,0) -- (2,0);
            \draw[black, thick] (-2,0) -- (-0.5,-1);
            \draw[black, thick] (-0.5,-1) -- (2,0);
            \draw[blue, thick] (-2,-2) -- (2,-2);
            \draw[blue, thick] (-2,-2) -- (-0.5,-3);
            \draw[blue, thick] (-0.5,-3) -- (2,-2);     
            \draw[black, thick] (-2,2) -- (-2,0);
            \draw[black, thick] (-2,0) -- (-2,-2);
            \draw[black, thick] (2,2) -- (2,0);
            \draw[black, thick] (2,0) -- (2,-2);
            \draw[black, thick] (-0.5,1) -- (-0.5,-1);
            \draw[black, thick] (-0.5,-1) -- (-0.5,-3);
            \draw[black, dotted, thick] (-2,2) -- (2,0);
            \draw[black, thick] (-2,2) -- (-0.5,-1);
            \draw[black, thick] (2,2) -- (-0.5, -1);
            \draw[black, thick] (-2,0) -- (-0.5, -3);
            \draw[black, thick, dotted] (-2,0) -- (2,-2);
            \draw[black, thick] (2,0) -- (-0.5, -3);
        \end{tikzpicture}
        \caption{The triangulated cylinder from Example \ref{ex:cylinder}. 
        }
        \label{fig:cylinder}
    \end{figure}
\end{example}
 
In the example, a non-trivial element of $\Cok{ \Phi_k}$ corresponded to a $(k+1)$-dimensional surface whose boundary was the union of a $k$-cycle in $\A$ and a $k$-cycle in $\B$ representing non-trivial elements of $H_k(\A) $ and $H_k(\B)$, respectively.  Such a surface is called a cobordism \cite{stong} thus motivating our terminology. The fact that the cobordism is not counted when the cylinder is closed off justifies the term "open cobordism". A similar interpretation can be given for arbitrary spaces. The rest of this section is devoted to making that rigorous.

\subsection{Definition of $\Phi$ in terms of relative homology}

We have defined $\Phi $ in terms of kernels to be able to draw parallels to the notion of kernel persistence from \cite{kernel-persistence}. However, for some of the proofs and interpretations below, it may be convenient to think of kernels in terms of relative homology. 

Consider the following long exact sequence
\begin{equation*}
	\begin{tikzcd}
		\arrow[r] & H_{*+1}(\X)  \arrow[r, "j_{*}^\A"] & H_{*+1}(\X,\A) \arrow[r, "\partial^\A"]  & H_*(\A)  \arrow[r, "\iota_*^\A"] & H_*(\X) \arrow[r] &.{}
	\end{tikzcd}
\end{equation*}	
Exactness implies that 
\begin{equation}\label{eq:ker_iso}
\Ker{\iota^\A_*} \cong \Img{\partial^\A} \cong H_{*+1}(\X,\A)/\Img{j^\A_{*}}.
\end{equation}
Thus, we will often think of an element in $\Ker{\iota^\A_k}$ as represented by a relative cycle, i.e.\ a $(k+1)$-chain with boundary in $\A$. Such a representative is unique up to addition of cycles in $\X$ and $(k+1)$-chains completely contained in $\A$. We denote the space of relative $(k+1)$-cycles by $Z_{k+1}(\X,\A)$.

 We also have a map of exact sequences induced by the inclusion $\A\hookrightarrow \A\cup \B$
\begin{equation*}
\begin{tikzcd}
 \arrow[r] & H_{*+1}(\X) \arrow[d] \arrow[r, "j_{*}^\A"] & H_{*+1}(\X,\A) \arrow[r, "\partial^\A"] \arrow[d] & H_*(\A) \arrow[d] \arrow[r, "\iota_*^\A"] & H_*(\X) \arrow[d]\arrow[r] &{}\\
 \arrow[r] &H_{*+1}(\X) \arrow[r, "j_{*}^{\A\cup \B}"] &H_{*+1}(\X,\A\cup \B) \arrow[r, "\partial^{\A\cup \B}"] &H_*(\A\cup \B) \arrow[r, "\iota_*^{\A\cup \B}"] &H_*(\X) \arrow[r] &{}
\end{tikzcd}
\end{equation*}
and a corresponding diagram for $\B$. In combination with \eqref{eq:ker_iso}, this shows that
the cokernel of $\Phi$ is isomorphic to the cokernel of 
$$H_{*+1}(\X,\A)/\Img{j_{*}^\A}\oplus H_{*+1}(\X,\B)/\Img{j_{*}^\B} \to H_{*+1}(\X,\A\cup \B)/\Img{j_{*}^{\A\cup \B}}. $$
This is again the same as the cokernel of 
\begin{equation}\label{eq:rel-cok}
H_{*+1}(\X,\A)\oplus H_{*+1}(\X,\B) \to H_{*+1}(\X,\A\cup \B), 
\end{equation}
since  $j_{*}^{\A\cup \B}$ filters as $j_{*}^{\A\cup \B}:H_{*+1}(\X) \xrightarrow{j_*^\A} H_{*+1}(\X,\A) \to H_{*+1}(\X,\A\cup \B)$ and hence $\Img{j_{*}^{\A\cup \B}}\subseteq \Img{H_{*+1}(\X,\A) \to H_{*+1}(\X,\A\cup \B)}$.

 Equation \eqref{eq:rel-cok} shows that a non-trivial class in the cokernel of $\Phi_k$ is represented by a $(k+1)$-chain with boundary in $\A \cup \B$ that cannot be written as a sum of chains with boundaries purely in $\A$ or $\B$. 
 In particular,
 the cells of the chain must form a $(k+1)$-surface with boundary in $\A\cup \B$ with at least one connected component intersecting both $\A$ and $\B$. See Figure \ref{fig:tunnels} for examples of what is counted or not counted as a tunnel.

\begin{figure}
    \centering
    \includegraphics[width=\linewidth]{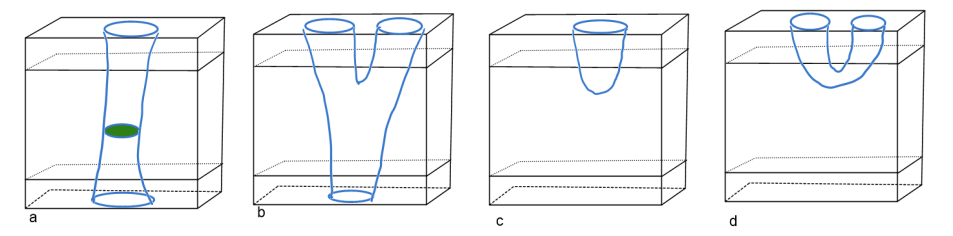}
    \caption{Examples of tunnels. The blue surface is $\X$, the part in the top slice is $\A$ and the part in the bottom slice is $\B$. In a), the tunnel is closed in the middle and hence not counted because it is a sum of two elements in $\Ker{\iota_*^\A}$ and $\Ker{\iota_*^\B}$. In b), there is 1 tunnel, and in c) and d), there are no tunnels since these are elements of $\Ker{\iota_*^\A}$.}
    \label{fig:tunnels}
\end{figure}

\subsection{More on the point cloud case} \label{sec:point_cloud}

For our application in chemistry, we consider a point cloud  $P\subseteq [0,1]^3$ representing the position of atoms. The union of closed balls of radius $r$ centered around the points in $P$ is denoted $P_r$.  Let $\X_r \subseteq P_r$ be the associated alpha-complex at fixed filtration value $r$, see e.g.\ \cite{edHar}. This is a  simplicial complex which deformation retracts onto $P_r$. To model the top and bottom of $\X_r$, we let  $A= [0,1]^2\times [1-\varepsilon,1] $ and $B=[0,1]^2\times [0,\varepsilon] $ be slices of $[0,1]^3$ at the top and bottom, respectively. 
The subcomplexes $\A_r$ and $\B_r$ are the simplicial complexes formed by all simplices of $\X_r$ with vertices in $P\cap A$ and $P\cap B$, respectively.  
A representative of a non-trivial element in $\Cok {\Phi_1}$ is a 2-dimensional simplicial surface in $\X_r$ whose boundary is the union of non-trivial 1-cycles in $\A_r $ and  $\B_r$. There must be a free passage from $A$ to $B$ through the surface, since otherwise, the element would be a sum of an element from  $H_{2}(\X_r,\A_r)$ and one from $ H_{2}(\X_r,\B_r)$ and thus trivial in $\Cok{ \Phi_1}$ by \eqref{eq:rel-cok}.  We think of this passage as a tunnel. A path through this passage can always be pushed to the complement of $P_r$, which means that it has distance at least $r$ to all points in $P$. Equivalently, one can roll a ball of radius $r$ along the path. Thus, we say that the tunnel has thickness at least  $r$.  By adding closed cycles and relative cycles with boundary in only one of $\A_r$ and $\B_r$, we can obtain an equivalent representative for the class in $\Cok{ \Phi_1}$ that bounds one or more tunnels connecting $\A_r$ and $\B_r$. On the other hand, the boundary of such a tunnel will be represented by an element of $\Cok{ \Phi_1}$.  

 In order for a cobordism from $\A_r$ to $\B_r$ to be non-trivial in $\Cok{\Phi_1}$, its boundaries must be non-trivial in $H_*(\A_r)$ and $H_*(\B_r)$. Typically, this will happen when a tunnel exits at the top and botton of the box, but degenerate cases may occur, see Figure \ref{fig:degenerate}. In fact, the definition of tunnels is quite sensitive to the choice of $\A_r$ and $\B_r$. Choosing the thickness $\varepsilon$ of the top and bottom slices is a non-trivial problem in practice. On the one hand, $\varepsilon $ must be large enough to contain the boundary of the tunnel. On the other hand, a very large $\varepsilon$ increases the risk of degenerate tunnels. 

 \begin{figure}
    \centering
    \includegraphics[width=\linewidth]{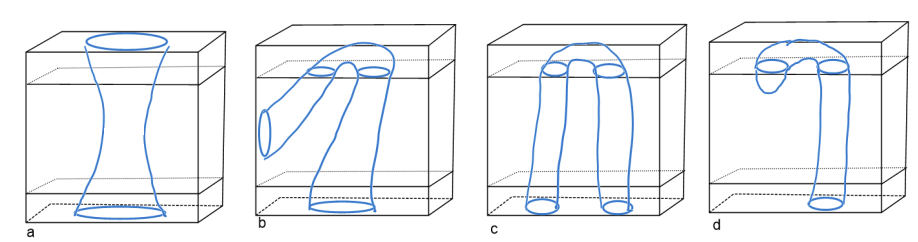}
    \caption{Examples of cobordisms in $\Cok{\Phi_1}$. The blue surface is $\X$, the part in the top slice is $\A$ and the part in the bottom slice is $\B$. The typical case is a) where the tunnel exits through the top and bottom. In both b) and c) we count 1 tunnel, while d) is not counted as a tunnel.  
    }
    \label{fig:degenerate}
\end{figure}

Since we are really trying to model the void space between atoms, a natural alternative approach would be to model the void space directly via the dual Voronoi complex. Indeed, recall the duality between the Voronoi and Delauney complex associated to a point cloud $P$, see e.g.\ \cite{edHar}. Let $\mathbb{V}_{-r}$ denote the cell complex consisting of all Voronoi cells whose dual Delauney cell is not in $\X_r$. Then the complement of $P_r$ deformation retracts onto $\mathbb{V}_{-r}$, see Appendix \ref{appendix} for a rigorous statement. Thus, it is natural to think of a tunnel as a path from top to bottom in $\mathbb{V}_{-r}$ (with a suitable definition of top and bottom in $\mathbb{V}_{-r}$). 
Such tunnels can be identified by computing $\Cok {\Phi_0}$ for $\mathbb{V}_{-r}$. However, since working with the alpha-complex is most common in topological data analysis, we use the alpha-complex formulation in this paper.

\section{$\Cok \Phi$ as persistence module} \label{sec:persistence} 

Assume that $\X$ is a filtered cell complex with a filter map $f$ that  to each cell $\sigma$  in $\X$ assigns a real number $f(\sigma)$ such that $\tau \subseteq \sigma$ implies $f(\tau) \leq f(\sigma)$. We assume for simplicity that $f$ is injective. The filter map defines a filtration of $\X$ into cell complexes $\X_r = f^{-1}((-\infty,r])$ with inclusion maps $h_{r,s}:\X_r \to \X_s$ for $r<s$ inducing maps on homology $(h_{r,s})_*:H_*(\X_r)\to H_*(\X_s)$. For ease of notation, we assume that the image of $f$ is $\{1,\ldots,N\}$. Since the topology only changes when a new cell is added to $\X_r$, all relevant information is contained in the following persistence module
$$H_*(\X_0) \xrightarrow{(h_{0,1})_*} H_*(\X_1) \xrightarrow{(h_{1,2})_*} \dotsm \xrightarrow{(h_{N-1,N})_*} H_*(\X_N).$$
By the injectivity assumption on $f$,  the dimension can only jump up or down by 1 at each step. In this case, we say that a birth or death happens, respectively. The filtration value when a birth or death happens is referred to as a birth or death time. A death time $d$ is paired with the birth time $b$ of the earliest born class that dies at time $d$. Unpaired birth times correspond to infinite cycles, i.e.\ classes that remain alive at time $N$.

 The filter map $f$ induces a filter map on $\A$, $\B$ and $\A \cup \B$, which means that the inclusion maps induce maps of persistence modules
\begin{equation*}
	\begin{tikzcd}
		H_*(\X_0) \arrow[r] & H_*(\X_1) \arrow[r ] &\dots \arrow[r] & H_{*}(\X_N) \\  
  H_*(\A_0) \arrow[r] \arrow[u]& H_*(\A_1) \arrow[u]\arrow[r ] &\dots \arrow[r] & H_{*}(\A_N)\arrow[u]
	\end{tikzcd}
\end{equation*}	
etc. As noted in \cite{kernel-persistence}, the kernels of the vertical maps again form a persistence module, which we refer to as $\Ker{ \iota_*^\A}$, with filtration $\Ker{ \iota_*^\A}_i = \Ker{\iota_*^{\A_i}: H_*(\A_i)\to H_*(\X_i)}$ for $i=1,\ldots,N$. Similarly, the images and cokernels of the vertical maps again form persistence modules. Extending this reasoning, we see that also $\Cok \Phi$ becomes a persistence module with filtration $\Cok \Phi_i$, $i=1,\ldots,N$. The birth and death times of $\Cok \Phi$ are identified in the next section.

\begin{example}
For the point cloud example in Section \ref{sec:point_cloud},  the Delauney complex is a filtered by alpha-filtration. The filtered complex at filtration value $r$ becomes exactly $\X_r$. If a tunnel is alive at filtration value $r$, it means the thickness of the tunnel is at least $r$. The death time of a tunnel corresponds to the largest radius of a ball that can roll through the tunnel without touching the original point cloud $P$. We call this the bottleneck thickness. The birth time of a tunnel is the first time it becomes disconnected from an existing tunnel or from the unbounded component.
\end{example}
    
\subsection{Births and deaths in $\Cok \Phi$}

The theorem below describes how the birth and death times in $\Cok \Phi $ are completely determined by those in $\Ker{\iota^{\A}_{\ast}}$, $\Ker{\iota^{\B}_{\ast}}$ and $\Ker{\iota^{\A\cup \B}_{\ast}}$. Illustrations of the different cases are given in Figure \ref{fig:cases}. 
The theorem does not give any information on how births and deaths are paired. An algorithm for pairing birth and death times is given in Section~\ref{sec:alg}.

\begin{figure}
    \centering
    \includegraphics[width=\linewidth]{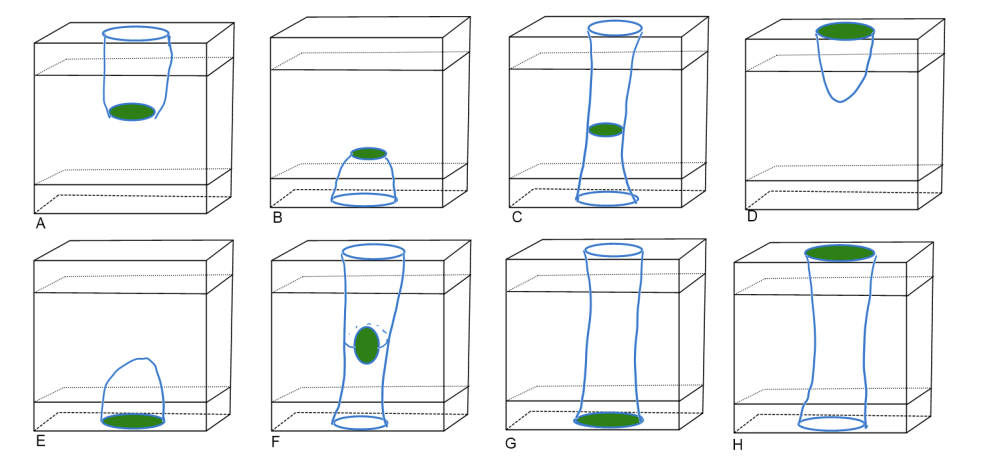}
    \caption{The 8 cases in Table \ref{tab:birth-death-events}. The cell that is inserted when a birth or death happens is shown in green. }
    \label{fig:cases}
\end{figure}
    
    \begin{theorem}\label{thm1}
 A class in $\Cok{\Phi}$ is born when a class in $\Ker{\iota^{\A\cup\B}_{\ast}}$ is born and there is no class born in $\Ker{\iota^{\A}_{\ast}}$ and $\Ker{\iota^{\B}_{\ast}}$ (Case F in \Cref{tab:birth-death-events}). A class in $\Cok{\Phi}$ dies when a class is born in $\Ker{\iota^{\A}_{\ast}}$ and $\Ker{\iota^{\B}_{\ast}}$, and hence in $\Ker{\iota^{\A\cup\B}_{\ast}}$ (Case C in \Cref{tab:birth-death-events}), or when a class is born in precisely one of $\Ker{\iota^{\A}_{\ast}}$ and $\Ker{\iota^{\B}_{\ast}}$, but not in $\Ker{\iota^{\A\cup\B}_{\ast}}$ (Case G and H in \Cref{tab:birth-death-events}).
    \end{theorem}

For the proof, we recall from \cite{kernel-persistence} how births and deaths happen in $\Ker{\iota^{\A}_{k}}$. Say a class $[c]\in \Ker{\iota^{\A}_{k}}_b$ is born at time $b$. The representative $c$ must be a non-trivial $k$-cycle in $\A_{b-1}$ which becomes a boundary in $\X_b$ when adding a ($k+1$)-cell  $\Delta\in \X_b\backslash \A_b$ of filtration value $b$. That is, there is a $(k+1)$-dimensional chain $c_1$ in $\X_{b-1}$ with $\partial (c_1+\Delta)= c$. 
		Moreover, no $(k+1)$-chain $c_2$ in $\X_{b-1}$ with $\partial c_2 = c$ existed before adding $\Delta$ (because then $[c]$ would have been born earlier). Note that $\Delta$ cannot lie in $\A_b$, since otherwise $c+\partial \Delta = \partial c_1 $ and hence $[c+\partial \Delta] \in \Ker{\iota^{\A}_{\ast}}_{b-1}$ before adding $\Delta$. Hence, after adding $\Delta$, $[c]=[c+\partial \Delta]$ is not new in $\Ker{\iota^{\A}_{\ast}}_b$. The only way a class $[c]\in \Ker{\iota^{\A}_{\ast}} \subseteq H_\ast(\A)$ can die is by being killed in $H_\ast(\A)$.

    \begin{proof}
        We will argue, that there are only the possible birth-death combinations shown in \Cref{tab:birth-death-events}. First of all, note that $\Phi_k$ is injective by \eqref{eq:sum}. Thus, a birth or death in $\Cok{\Phi_k}$ can only happen when the dimension of the domain or codomain of $\Phi_k$ changes, i.e. when a birth or death happens there. Moreover, the assumption that $f$ is injective ensures that at most one birth or death can happen at a given time in each of $\Ker{\iota^{\A}_{\ast}}$, $\Ker{\iota^{\B}_{\ast}}$, and $\Ker{\iota_{\A\cup\B}^{\ast}}$.

		Suppose a class $[c]\in \Ker{\iota^{\A}_{\ast}}_b$ is born at time $b$. 
		This must happen by adding a $(k+1)$-cell $\Delta\in \X_b\backslash \A_b$. There are three possibilities:
  
		\begin{itemize}		
		 \item[(i)] If $\Delta\in \X_b\backslash (\A_b\cup \B_b)$, this implies a birth of $[c]$ in  $\Ker{\iota^{\A\cup\B}_{\ast}}_b$.  If there is no simultaneous birth in $\Ker{\iota^{\B}_{\ast}}_b$, this is Case A.
		
		\item[(ii)]	If  $\Delta\in \X_b\backslash (\A_b\cup \B_b)$ also leads to a birth of  $[c']\in \Ker{\iota^{\B}_{\ast}}_b$, then there must exist $(k+1)$-chains $c_1$ and $c_1'$ in $\X_{b-1}$ with $c= \partial(c_1 + \Delta)$ and $c'= \partial(c_1' + \Delta)$. Thus, $[c+c'] = [\partial(c_1 + c_1')]$ must be a  non-trivial element in $\Ker{\iota^{\A\cup\B}_{\ast}}_{b-1}$ so that there is only one birth in $\Cok{\Phi_k}_b$. This is Case C.		
				
		\item[(iii)] If $\Delta\in \B_b$, there exists a $(k+1)$-chain $c_1$ in $\X_{b-1}$  such that $\partial (c_1 + \Delta) = c$, that is, $[c + \partial \Delta] \in \Ker{\iota^{\A\cup\B}_{\ast}}_{b-1}$. This represents a non-trivial element of  $\Cok{\Phi_k}_{b-1}$ because $[c] $ is not yet born in $ \Ker{\iota^{\A}_{\ast}}_{b-1}$. When adding $\Delta$, $[c]$ is born in $ \Ker{\iota^{\A}_{\ast}}_{b}$, while $[\partial \Delta]=0$ in $ \Ker{\iota^{\B}_{\ast}}_{b}$.  Hence $[c+\partial \Delta]$ is in the image of $\Phi_k$ and hence killed by $\Delta$ in  $\Cok{\Phi_k}_b$. This is Case G. 
  		\end{itemize}

		 A death in $\Ker{\iota^{\A}_{\ast}}_d$ implies a death in $\Ker{\iota^{\A\cup\B}_{\ast}}_d$ by injectivity of $\Phi_k$. This is Case D. Injectivity of $f$ and disjointness of $\A $ and $\B$ implies that a death cannot happen simultaneously in $\Ker{\iota^{\A}_{\ast}}_d$ and $\Ker{\iota^{\B}_{\ast}}_d$.
		
		Finally, a class can be born in $\Ker{\iota^{\A\cup\B}_{\ast}}_b$  without being the image of $\Phi_k$ as in Case F. However, a death in $\Ker{\iota^{\A\cup\B}_{\ast}}_d$ cannot happen without a death in $\Ker{\iota^{\A}_{\ast}}_d$ or $\Ker{\iota^{\B}_{\ast}}_d$. Indeed, a death of  $[c]\in\Ker{\iota^{\A\cup\B}_{\ast}}_d$ would mean that $[c]$ becomes trivial in  $H_{\ast}(\A_d)\oplus H_{\ast}(\B_d)$ when inserting $\Delta$. Since $\A_d$ and $\B_d$ are disjoint, this  would require $[c]$ to be in one of the summands. Hence $[c] $ would also die in $\Ker{\iota^{\A}_{\ast}}_d$ or $ \Ker{\iota^{\B}_{\ast}}_d$. 
		
		The remaining cases B, E and H 
         are found as above by replacing $\Ker{\iota^{\A}_{\ast}}$ by $\Ker{\iota^{\B}_{\ast}}$.
        
        \begin{table}
            \centering
            			\begin{tabular}{|r|c|c|c|c|c|c|c|c|}\hline
				Case: & A & B & C & D & E & F & G & H\\ \hline
				$\Ker{\iota^{\A}_{\ast}}$ & b & - & b & d & - & - & b & -\\ \hline
				$\Ker{\iota^{\B}_{\ast}}$ & - & b & b & - & d & - & - & b\\ \hline
				$\Ker{\iota^{\A\cup\B}_{\ast}}$ & b & b & b & d & d & b & - & -\\ \hline
				$\Cok{\Phi}$ & - & - & d & - & - & b & d & d\\ \hline
			\end{tabular}
            
            \caption{Possible combinations of birth-death events.
              }\label{tab:birth-death-events}
        \end{table}

    \end{proof}

\section{Algorithms}\label{sec:alg}

\subsection{The algorithm for finding kernel persistence}\label{sec:ker}

We first recall the algorithm for computing kernel persistence from \cite{kernel-persistence}. This can be used for finding births and deaths  in $\Cok \Phi$ using Table \ref{tab:birth-death-events}, and it will be the foundation for and share many similarities with the algorithm for pairing births and deaths in $\Cok \Phi$ in Section \ref{sec:pairing}. The precise algorithm is given in Algorithm \ref{alg:kernel}. The algorithm makes use of the standard column reduction algorithm that uses column addition from left to right to push the lowest 1 of each column as far up as possible, see \cite[Sec. VII.1]{edHar} for details.

\begin{algorithm}

\KwData{The boundary matrix $D^{\X}$ of $\X$ with the rows and columns ordered by $f$.}
\KwResult{Pairs of birth and death times in $\Ker{\iota^\A_*} $.}

Obtain a new matrix $D_{im}^{\A}$ by reordering the rows of $D^{\X}$, but not the columns, such that the cells in $\A$ come first followed by cells in $\X\setminus\A$, and rows ordered by filtration value within each block. 

Run the standard column reduction algorithm  on $D_{im}^{\A}$ to obtain matrices $R_{im}^{\A}$ and $V_{im}^{\A}$ such that
    $R_{im}^{\A} = D_{im}^{\A} V_{im}^{\A}$.
    
 The column of a cell $\tau \notin \A$ with lowest 1 in the $\A$-block of  $R_{im}^{\A}$ represents a birth in $\Ker{\iota^\A_*} $. The corresponding column in $V_{im}^{\A}$ stores a representative of the relative cycle born.
 
 Construct a new matrix $D_{ker}^{\A}$ by removing the columns in $V_{im}^{\A}$ that do not store cycles.
 
Reorder the rows of $D_{ker}^{\A}$ so that the cells in $\A$ come first, followed by cells in $\X\setminus\A$, both ordered by $f$.

Reduce $D_{ker}^{\A}$ to obtain $R_{ker}^{\A}$, $R_{ker}^{\A} = D_{ker}^{\A} V_{ker}^{\A}$.

 The column of $\tau$ in $R_{ker}^{\A}$ represents a death if $\tau \in \A$ and its lowest 1 cannot be pushed to $\A$. The row of the lowest 1 gives the corresponding  birth time.
 \caption{Algorithm for kernel persistence}\label{alg:kernel}
\end{algorithm}

Using the relative homology formulation, it is easy to see what the reduction of $D_{Im}^\A$ does  (note the analogy with the algorithm for finding relative homology in \cite{edHar}): By ordering the rows such that cells in $A$ come first, the reduction algorithm always tries to push the lowest 1 to $\A$ if possible. Thus, in the column of a $(k+1)$-cell $\tau$ we try to add $(k+1)$-cells with lower filtration value to $\tau $ to form a $(k+1)$-chain with boundary in $\A$. This is possible if the column of  $\tau$  in $R_{Im}^\A$ is zero (in this case $\tau $ creates a cycle in $\X$) or has a lowest 1 in $\A$ (in this case $\tau$ creates a relative cycle with non-trivial boundary in $\A$). This relative cycle is non-trivial in $H_{k+1}(\X,\A)/\Img{j^\A_{k+1}}$ if $\tau \notin \A$ (if $\tau \in \A$ the whole relative cycle would lie in $\A$) and  the $\tau$-column in $R_{Im}^\A$ is non-zero (such that it is not a closed cycle). Thus, $\tau $ gives birth to a new element of $H_{k+1}(\X,\A)/\Img{j^\A_{k+1}}$ when $\tau\notin \A$ and it has a lowest 1 in the $\A$-block of $R_{Im}^\A$. 

A death of a class in $\Ker{\iota^\A_*} \subseteq H_*(\A)$ can only happen if the class becomes zero in $H_*(\A)$. So it must happen by adding a cell $\tau \in \A$ that kills a cycle in $\A$. This cycle should represent something in $\Ker{\iota^\A_*}$, i.e. it must be the boundary of a relative cycle. Adding $\tau$ closes this relative cycle. 
To pair births and deaths, we use $D_{Ker}^\A$. The columns represent cycles in $\X$, and a column corresponding to $\tau \in \A$ represents a death if its lowest 1 cannot be pushed to $\A$, that is, $\tau$ closes a non-trivial relative cycle. The  reduction of $D_{Ker}^\A$ tries to find equivalent representatives such that the largest filtration value of the cells in $\X - \A$ is as small as possible, i.e.\ the earliest born relative cycle that is killed by $\tau$.

\subsection{Algorithm for pairing births and deaths in $\Cok \Phi$}\label{sec:pairing}

Algorithm \ref{alg:pair} gives an algorithm for pairing birth and death times in $\Cok\Phi$. The algorithm uses the notation $M[\tau]$ for the column  indexed by $\tau$ in a matrix $M$.

\begin{algorithm}
\KwData{Birth and death times in $\Ker{\iota^\A_*}$, $\Ker{\iota^\B_*}$ and $\Ker{\iota^{\A\cup \B}_*}$ and matrices $R^{\A}_{Im}=D^\A_{Im} V^\A_{Im}$ and  $R^\B_{Im}=D^\B_{Im} V^\B_{Im}$ computed using Algorithm \ref{alg:kernel}.}
\KwResult{Birth and death pairs in $\Cok \Phi$.}

Reorder the rows of $V^\A_{Im}$, and $V^\B_{Im}$ consistently such that $\A\cup \B$ comes before $\X\backslash (\A\cup \B)$ and the rows in the $\X\backslash (\A\cup \B)$ block are ordered by filtration value.

Define a new matrix $D^\Phi$ as the columns of $V^\A_{Im}$ corresponding to zero-columns in $R^\A_{Im}$.\label{(i)}

\For{$\tau \in \X$}{\label{for}
    \If{$R^{\A}_{Im}[\tau] $ is non-zero and  its lowest 1 is in $\A$}{insert $V^\A_{Im}[\tau]$ into $D^\Phi$. \label{(ii)}}
 
    \If{$R^{\B}_{Im}[\tau]$ is non-zero and  its lowest 1 is in $\B$}{\label{(iii)_start}

        \eIf{$D^\Phi$ has no column corresponding to $\tau$}{insert $V^\B_{Im}[\tau]$ into $D^\Phi$. \label{(iii)_1}} 
        {\eIf{$\tau\in \A$}{insert $V^\B_{Im}[\tau]$ into $D^\Phi$ after the existing $\tau$-column. \label{(iii)_2}}
        {insert $V^\B_{Im}[\tau]$ into $D^\Phi$ before the existing $\tau$-column. \label{(iii)_3}}}
    }
}

Run the standard column reduction algorithm on $D^\Phi$ to form $R^\Phi=D^\Phi V^\Phi$.

Double columns in $D^\Phi$ correspond to a death in $\Cok \Phi$.  
The row with the lowest 1 in the second column gives the corresponding birth time.

Compute the birth and death times in $\Cok{\Phi}$ from \Cref{{tab:birth-death-events}}.

If there are birth times in $\Cok \Phi$ that are not matched to a death by the algorithm, it corresponds to an infinite cycle. 

    \caption{Algorithm for pairing birth and death times}\label{alg:pair}
\end{algorithm}

\begin{theorem}\label{thm2}
Algorithm \ref{alg:pair} correctly pairs birth and death times in $\Cok \Phi$. 
\end{theorem}

\begin{proof}
	In the following, when we think of representatives for kernels, we use the relative homology formulation \eqref{eq:ker_iso}. E.g.\ we think of an element of $\Ker{\iota_*^\A}$ as represented by a chain with boundary in $\A$, i.e.\ an element of the space of relative cycles $Z_{*+1}(\X,\A)$.  	
    
    The column of some $\tau$ added to $D^\Phi$ in Line \ref{(i)} of the algorithm corresponds to a cycle in $\X$ that $\tau$ gives birth to. Thus, the columns added in Line \ref{(i)} correspond to representatives of a basis for the space of cycles in $\X$ (such a basis  could  have been identified from any of the matrices $R^\X$, $R^\A_{Im}$, or $R^\B_{Im}$, but the specific basis may vary). A basis for the relative cycles with boundary in $\A$  is given by the columns  added in Line \ref{(i)}  and \ref{(ii)}. Similarly, a basis for the relative cycles with boundary in $\B$ is given by the columns added in Line \ref{(i)}  and \ref{(iii)_1} -- \ref{(iii)_3}.

    The column of some $\tau$ added in Line \ref{(ii)} will represent a birth in $\Ker{\iota_*^\A}$ if $\tau \notin \A$. Indeed, as explained in Section \ref{sec:ker}, a birth happens in $\Ker{\iota_*^\A}$ when adding a cell $\tau \notin \A$, whose column in $R^\A_{Im}$  is non-zero and has its lowest 1  in one of the $\A$-rows. If additionally $\tau \notin \A\cup \B$, then it will represent a birth in $\Ker{\iota_*^{\A\cup \B}}$ by the same argument. (Specifically, to check that its lowest 1 in $R^{\A\cup \B}_{Im}$ must be in the $\A\cup \B$-block, we note that if it is possible to push the lowest 1 of a column to $\A$ in  $R_{Im}^\A$, then it also possible to push it to $\A\cup \B$ in $R_{Im}^{\A\cup \B}$.) The chain stored in the column of $\tau$ in $V_{Im}^\A$ is a cycle in $Z_{*+1}(\X,\A)$ born by $\tau$.
	
	Similarly, the column of some $\tau$ added in Line \ref{(iii)_1} -- \ref{(iii)_3} will represent a birth in $\Ker{\iota_*^\B}$ if $\tau \notin \B$ and in $\Ker{\iota_*^{\A\cup \B}}$ if $\tau \notin \A\cup \B$. By Theorem \ref{thm1}, a death in $\Cok \Phi$ can only happen when a birth takes place in at least one of $\Ker{\iota_*^\A}$ and $\Ker{\iota_*^\B}$, i.e.\ when a column is added somewhere in Line \ref{(ii)}--\ref{(iii)_3}. It is now straightforward to check which cases can occur: 
	\begin{itemize}
		\item[1.] If a $\tau $ column is added in Line \ref{(i)}, then it creates a cycle in $\X$. This does not change $\Cok \Phi$. A $\tau$ column cannot be added in any of the following lines, because the zero-columns of $R^\A_{Im}$, and $R^\B_{Im}$ are the same.
		\item[2.] If $\tau \notin \A\cup \B$ and a $\tau$ column is added in both Line \ref{(ii)} and in Line \ref{(iii)_3}, then we are in Case C. The two $\tau $ columns will represent newborn elements of $\Ker{\iota_*^\A}$ and $  \Ker{\iota_*^\B}$, respectively, while their sum will represent the dying element in $\Cok\Phi$.
		\item[3.] If $\tau \notin \A\cup \B$ and the $\tau$ column is added only in Line \ref{(ii)}, then we are in Case A.
		\item[4.] If $\tau \notin \A\cup \B$ and the $\tau$ column is added only in Line \ref{(iii)_1}, then we are in Case B.
		\item[5.] If $\tau \in \A$ and the $\tau$ column is added in both Line \ref{(ii)} and \ref{(iii)_2}, then we are in Case H.  
        The column added in  Line \ref{(iii)_2} will represent a birth in $ \Ker{\iota_*^\B}$, while the column added in Line \ref{(ii)} will be a chain in $\A$ containing $\tau$, which is zero in $ \Ker{\iota_*^\A}$. The sum of the two columns will represent the dying element of $\Cok \Phi$.
		\item[6.] If $\tau \in \A$ and the $\tau$ column is added only in Line \ref{(ii)}, then no birth happens and we are not in any of the cases.
        \item[7.] The case $\tau \in \A$ and $\tau $ is only added in Line \ref{(iii)_2} cannot happen because then $\partial \tau$ would be trivial and hence $\tau$ would create a cycle and would only have been added in Line \ref{(i)}.
		\item[8.] If $\tau \in \B$ and the $\tau$ column is added in both Line \ref{(ii)} and Line \ref{(iii)_3}, then we are in Case G.  
        The column added in  Line \ref{(ii)} will represent something in  $ \Ker{\iota_*^\A}$, while the column added in Line \ref{(ii)} will be a chain in $\B$ including $\tau$. The sum of the two columns will represent the dying element of $\Cok \Phi$.
		\item[9.] If $\tau \in \B$ and the $\tau$ column is added only in  Line \ref{(iii)_1}, then no birth happens and we are not in any of the cases.
        \item[10.] Again, the case $\tau \in \B$ and $\tau $ is only added in Line \ref{(ii)} does not happen.
	\end{itemize}
	To summarize, it is exactly the $\tau $ that appear twice that correspond to Case C, G and H, where a death happens in $\Cok\Phi$. The sum of the two columns stores a representative of the element in $\Cok \Phi$ that dies. It follows from \eqref{eq:rel-cok} that such a representative is unique up to addition of relative cycles with boundary purely in either $\A$ or $\B$. 
	
	A basis for the relative cycles with boundary purely in $\A$  is added in Line \ref{(i)} + \ref{(ii)}. A basis for the relative cycles with boundary purely in $\B$  is added in Line \ref{(i)} + \ref{(iii)_1}--\ref{(iii)_3}.  Thus, by adding the columns to the left of $\tau $ we can get all equivalent representations of the tunnel that dies at time  $\tau$. The reduction algorithm tries to push the lowest 1 as far up as possible finding the earliest born representation.

\end{proof}

\subsection{Discussion of representatives}

Finding representatives of cobordism equivalence classes is relevant in chemistry, where such representatives carry information about the atomic structure around tunnels.
Theorem \ref{thm2} tells us, which cells give birth and death to a cobordism. For the point cloud in Section \ref{sec:point_cloud}, the birth cell is the one that first seperates the tunnel from an existing tunnel or from the unbounded component. The death cell is the one that closes the tunnel off at its most narrow passage. Its filtration value is the bottleneck width of the tunnel. From Algorithm \ref{alg:pair} we further get the following representations of cobordisms.

\begin{corollary}
Algorithm \ref{alg:pair} provides the following representatives of cobordisms:
\begin{itemize}
 \item[1.]  If a cobordism is killed by $\tau$, the sum of the two $\tau$-columns in $D^\Phi$ provides a representative of the cobordism that is alive just before death. 
	\item[2.]  If a cobordism is killed by $\tau$, the second $\tau$-column in $R^\Phi$ gives a representative of the cobordism when it was born. 
\item[3.] A representative for a cobordism that lives forever can be read off from the birth column in $V_{im}^{\A\cup \B}$.
\end{itemize}
\end{corollary}

\begin{proof}
1. was  observed in the proof of Theorem \ref{thm2}.
 To realize 2., suppose we run the reduction algorithm. If $\tau\in\A\cup \B$ corresponds to a double column, the second column stores a representative of the tunnel to begin with. If $\tau\notin\A\cup\B$ corresponds to a double column, when the algorithm gets to the second column, it will start by adding the first $\tau$ column since both have their lowest 1 in the $\tau$ row. The second column now stores a representative of the tunnel that dies at time $\tau$. By adding columns to the left of $\tau$, we can get all equivalent representations of the tunnel. By pushing the lowest 1 as far up as possible, we get a representative that is born as early as possible. The lowest 1 after the reduction marks the birth cell and the whole column gives a representation of the tunnel at birth. 
 \end{proof}

We note that the representatives found by Algorithm \ref{alg:pair} for the birth and death tunnel between $\A $ and $\B$ is by no means unique and there is no guarantee that the algorithm finds the most natural choice of representative. For instance, for the point cloud in Section~\ref{sec:point_cloud}, 
the representative needs not form the surface of the tunnel. There might be some cells inside the representative.

\section{Discussion and future directions}

In this paper, we have considered surfaces connecting two spaces. We believe that this could easily be extended to three or more spaces by an iterative procedure. We leave this for future investigation.

As mentioned in Section \ref{sec:point_cloud}, in the specific case of a point cloud $P$, it might make sense to model the Voronoi dual complex $\mathbb{V}_{-r}$ and identify tunnels by computing $\Cok{\Phi_0}$. This would have the advantage that the representatives found by the algorithm are paths connecting top and bottom, which would be more easily interpretable. We have chosen to work with the alpha-complex in this paper since this is most common, however, the dual perspective seems promising. 

The results in the point cloud case are quite sensitive to the thickness of top and bottom slices. If they are too thin, they are unlikely to contain the boundary of a tunnel, but if they are too thick, 
 degenerate cases from Figure \ref{fig:degenerate} are more likely to occur. One way of solving this problem is to choose thick slices in top and bottom and remove 2-cells completely contained in top or bottom 
 because these 2-cells are not necessary for forming the desired cobordisms (see Figure \ref{fig:degenerate}a), but they are necessary for forming the degenerate cases (see Figures \ref{fig:degenerate}b, c and d).
 Another possibility is to consider the Voronoi dual, where only a set of potential start and end points are needed as $\A$ and $\B$. 

Atomic point clouds are often simulated with periodic boundary conditions (i.e. on a 3-torus) to avoid boundary effects. Thus, it would make sense to identify the left and right side of the box and the front and back side (forming a 2-torus times an interval) before computing the alpha-complex. This would avoid problems with tunnels exiting the sides of the bounding box. However, new problems may occur due to the non-trivial topology of the torus.  For instance, a tunnel through a solid box as in Figure \ref{fig:torus} would not be counted since the top and bottom cycle would be trivial in $H_*(\A) $ and $H_*(\B)$, respectively. See also \cite{teresa,periodic} for more on the challenges with periodic boundary conditions.

\begin{figure}
    \centering
    \includegraphics[width=0.3\linewidth]{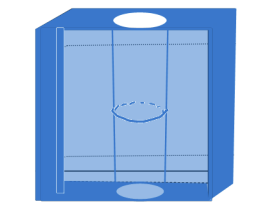}
    \caption{With periodic boundaries, this would not be counted as a tunnel.}
    \label{fig:torus}
\end{figure}

\section*{Acknowledgements}
Y. B. B. and L. F. were funded by the Independent Research
Fund Denmark, grant number 1026-00037. T. H. was partially supported by the European Research Council (ERC) Horizon 2020, grant number 788183.

\bibliographystyle{plain}
\bibliography{references.bib}

\appendix
\section{The Voronoi dual filtered complex}\label{appendix}

Consider a finite point cloud $ P\subseteq \R^d$ with points in general position, i.e. no $k+2$ points lie on a $(k-1)$-dimensional sphere.  
Let $\mathbb{V}$ be  the Voronoi complex associated with $P$ and let $\X$ be its dual Delauney complex, see \cite{edHar}. Each $k$-cell $\sigma$ in $\X$ corresponds to a dual $(d-k)$-cell $\sigma^*$ in $\mathbb{V}$ and vice versa. The complexes are dual in the sense that if $\tau \subseteq \sigma$ in $\X$ then $\sigma^* \subseteq \tau^*$ in $\mathbb{V}$. Let $f$ be the filter function on $\X$ that defines the alpha-complexes $\X_r = f^{-1}(\infty,r]$ for $r\geq 0$.
Define the filter function $f^*$ on $\mathbb{V}$ by $f^*(\sigma^*)=-f(\sigma)$. Then $\X$ and $\mathbb{V}$ are dual filtered complexes in the sense of \cite{dual}. In particular, the boundary matrix for $\mathbb{V}$ is found by transposing the boundary matrix for $\X$ and reversing the order of the rows and columns. Thus, also their persistent homology are closely related \cite{dual}.

For a fixed filtration value $r$, it is well-known that $P_r=\bigcup_{p\in P}\overline{B}_r(p)$, where $\overline{B}_r(p)$ denotes the closed ball around $p$ of radius $r$,  deformation retracts onto $\X_r$. We claim that also the complement of $ P_r$ is homotopy equivalent to $\mathbb{V}_{-r}=(f^*)^{-1}(-\infty,-r)$. To make this statement rigorous, we have to be a bit careful about what happens at infinity. The precise statement we show is the following.

\begin{theorem}
Let $B$ be a box in $\R^d$ containing $P$ such that the distance from $\partial B$ to $P$ is greater than $r$. Then, $B\setminus P_r$ 
deformation retracts onto $\partial B \cup (\mathbb{V}_{-r}\cap B)$.  Moreover, $S^d\setminus \X_r$ deformation retracts onto $\mathbb{V}_{-r} \cup \infty$, where ${S}^{d}$ is the one-point compactification of $\R^d$ and $\mathbb{V}_{-r} \cup \{\infty\}$ is the one-point compactification of $\mathbb{V}_{-r}$.
\end{theorem}

\begin{proof}
Let $\mathbb{V}^k$ denote the $k$-skeleton of $\mathbb{V}$. We construct the deformation retraction inductively by pushing $(B\cap \mathbb{V}^k)\setminus P_r$ onto $(\partial B\cup (B\cap (\mathbb{V}_{-r}\cup \mathbb{V}^{k-1})))\setminus P_r$.

We start with $k=d$. Each Voronoi $d$-cell $\sigma$ contains exactly one point $p_\sigma$ from $P$. Then $B\setminus P_r$ is the union of the sets $(B\cap \sigma) \setminus \overline{B}_r(p_\sigma)$ for all Voronoi $d$-cells $\sigma$. 
Since both $(B\cap \sigma)$ and $ \overline{B}_r(p_\sigma)$ are convex, we can push $(B\cap \sigma) \setminus \overline{B}_r(p_\sigma)$ radially along line segments from $p_\sigma$ towards $\partial (B\cap \sigma) \setminus \overline{B}_r(p_\sigma)$. Repeating this for each $\sigma$ yields a deformation retraction onto $(\partial B \cup (B\cap \mathbb{V}^{d-1}))\setminus P_r$. Note that $\mathbb{V}^{d-1}$ can also be written as $\mathbb{V}_{-r} \cup \mathbb{V}^{d-1}$ because $\mathbb{V}_{-r} \subseteq \mathbb{V}^{d-1}$ for $r\geq 0$ as all Delaunay vertices have filter value $0$ (and hence there are no $d$-dimensional Voronoi cells in $\mathbb{V}_{-r}$).

Now suppose we have deformed $(B\cap \mathbb{V})\setminus P_r$ onto $(\partial B \cup (B\cap (\mathbb{V}_{-r}\cup \mathbb{V}^{k})))\setminus P_r$. Let $\sigma$ be a $k$-cell of $\mathbb{V}$ whose interior is intersected by $P_r$ in at least one point $p_\sigma$. Any point in $P_r\cap \sigma$ has equal distance less than or equal to $r$ to the nearest $d-k+1$ points. Let $p$ be any of these points. Then,  $P_r\cap \sigma = \overline{B}_r(p)\cap \sigma$, which is again convex containing $p_\sigma$. Again, we may push $\sigma\setminus P_r$ radially from $p_\sigma$ towards $\partial (B\cap\sigma) \setminus P_r$. A $k$-cell $\sigma$ that is not intersected by $P_r$ is member of $\mathbb{V}_{-r}$ and hence left untouched. This completes the inductive step. When we have pushed $(B\cap \mathbb{V})\setminus P_r$ all the way to $(\partial B \cup (B\cap (\mathbb{V}_{-r}\cup \mathbb{V}^{0})))\setminus P_r$, we are done, because $(\partial B \cup (B\cap (\mathbb{V}_{-r}\cup \mathbb{V}^{0})))\setminus P_r = \partial B \cup (B\cap \mathbb{V}_{-r})$. To see this, note that $(\mathbb{V}_{-r}\cup \mathbb{V}^{0})\setminus P_r = \mathbb{V}_{-r}$, because the Voronoi vertices that are not in $P_r$ are in $\mathbb{V}_{-r}$.

The second statement is proved similarly by noting that pushing points radially out also works for the one-point compactification of the unbounded Voronoi cells.

\end{proof}

\end{document}